\documentclass{compositio}
\usepackage{amsmath,amssymb,amsthm,url}
\usepackage{comment,graphicx}

\numberwithin{equation}{section}

\newtheorem{thm}[equation]{Theorem}
\newtheorem*{thm*}{Theorem}

\newtheorem*{conj*}{Conjecture}
\newtheorem{lem}[equation]{Lemma}

\theoremstyle{remark}
\newtheorem{rmk}[equation]{Remark}

\newcommand{\field}[1]{\mathbb{#1}} 
\newcommand{\Q}{\field{Q}}
\newcommand{\C}{\field{C}}

\newcommand{\Z}{\field{Z}}

\newcommand{\F}{\field{F}}

\newcommand{\T}{\field{T}}

\newcommand{\quat}[2]{\displaystyle{\biggl(\frac{#1}{#2}\biggr)}}
\newcommand{\legen}[2]{\left(\frac{#1}{#2}\right)}

\newcommand{\cD}{\mathcal{D}}
\newcommand{\cH}{\mathcal{H}}
\newcommand{\cO}{\mathcal{O}}
\newcommand{\calO}{\mathcal{O}}

\newcommand{\fp}{\mathfrak{p}}
\newcommand{\fm}{\mathfrak{m}}
\newcommand{\frakp}{\mathfrak{p}}
\newcommand{\frakN}{\mathfrak{N}}

\newcommand{\fP}{\mathfrak{P}}

\newcommand{\slsh}[1]{\,|_{#1}\,}

\DeclareMathOperator{\Aut}{Aut}
\DeclareMathOperator{\End}{End}
\DeclareMathOperator{\Frob}{Frob}
\DeclareMathOperator{\Gal}{Gal}
\DeclareMathOperator{\GL}{GL}
\DeclareMathOperator{\SL}{SL}
\DeclareMathOperator{\PSL}{PSL}
\DeclareMathOperator{\PGL}{PGL}
\DeclareMathOperator{\tr}{tr}
\DeclareMathOperator{\lcm}{lcm} 
\DeclareMathOperator{\rmP}{P\!}

\newcommand{\spnew}{{}^{\text{new}}}

\begin{document}

\title[Nonsolvable number fields]{Nonsolvable number fields ramified \\ only at 3 and 5}
\author{Lassina Demb\'el\'e}
\email{l.dembele@warwick.ac.uk}
\address{Warwick Mathematics Institute, University of Warwick, Coventry CV4 7AL, United Kingdom}
\author{Matthew Greenberg}
\email{mgreenbe@math.ucalgary.ca}
\address{University of Calgary, 2500 University Drive NW, Calgary, AB, T2N 1N4, Canada}
\author{John Voight}
\email{jvoight@gmail.com}
\address{Department of Mathematics and Statistics, University of Vermont, 16 Col\-chester Ave, Burlington, VT 05401, USA}
\thanks{The first author was supported by a SFB/TR 45 of the Deutsche For\-schungsgemeinschaft research grant; the second author acknowledges the support of an NSERC discovery grant; the third author was supported by NSF grant DMS-0901971.}

\begin{abstract}
For $p=3$ and $p=5$, we exhibit a finite nonsolvable extension of $\Q$ which is ramified only at $p$, proving in the affirmative a conjecture of Gross.  Our construction involves explicit computations with Hilbert modular forms.
\end{abstract}

\maketitle\vspace*{12pt}

The study of Galois number fields with prescribed ramification remains a central question in number theory.  Class field theory, a triumph of early twentieth century algebraic number theory, provides a satisfactory way to understand solvable extensions of a number field.  To investigate nonsolvable extensions, the use of the modern techniques of arithmetic geometry is essential.

Implicit in his work on algebraic modular forms on groups of higher rank, Gross~\cite{gross1} proposed the following conjecture.

\begin{conj*}\label{conj:unramified} 
For any prime $p$, there exists a nonsolvable Galois number field ramified only at $p$.
\end{conj*}

Serre \cite{serre3,serre5} has answered this conjecture in the affirmative for primes $p \geq 11$ in his analysis of Galois representations associated to classical cusp forms, exhibiting an irreducible (continuous, odd) representation $\Gal(\overline{\Q}/\Q) \to \GL_2(\overline{\F}_p)$ ramified only at $p$.  Conversely, it is a consequence of the proof of Serre's conjecture by Khare and Wintenberger \cite{khare1} together with standard level lowering arguments that if $p \leq 7$ then any odd representation of the absolute Galois group $\Gal(\overline{\Q}/\Q)$ to $\GL_2(\overline{\F}_p)$ which is ramified only at $p$ is necessarily reducible (and thus solvable).  

Therefore, to find nonsolvable number fields which are ramified only at a prime $p \leq 7$ we are led to consider more general settings.  Computations by Lansky and Pollack \cite{LanskyPollack} predict the existence of a $G_2(\F_5)$-extension of $\Q$ ramified only at $5$; however, the general theory which would provably attach a Galois representation to such an automorphic form is absent (see Gross~\cite{gross2}).  Following a suggestion of Gross, the first author \cite{Dembele2} has recently constructed a nonsolvable extension ramified only at $p=2$ by instead enlarging the base field: he exhibits a Hilbert modular form of level $1$ and parallel weight $2$ over the totally real subfield $F=\Q(\zeta_{32})^+$ of $\Q(\zeta_{32})$ whose mod $2$ Galois representation cuts out a number field with Galois group $8 \cdot \SL_2(\F_{2^8})^2$ ramified only at $2$.  
(We review this construction in Section~\ref{sec:review-p=2}, and refer to work of the first author~\cite{Dembele2} for additional details.)  Here and throughout, given $n \in \Z_{>0}$ and a group $G$ we write $n\cdot G$ for a group isomorphic to a semidirect product $n\cdot G \cong G \rtimes \Z/n\Z$.

In this paper, we pursue these ideas further and prove that Gross' conjecture is true for $p=3$ (Theorem \ref{peq3}) and $p=5$ (Theorem \ref{peq5}).  

\begin{thm*} \label{mainthm}
For $p=3$ and $p=5$, there exist nonsolvable Galois number fields ramified only at $p$.
\end{thm*}

We construct these Galois extensions by looking at the reduction of Hilbert modular forms of parallel weight $2$ defined over a totally real field $F$, where $F=\Q(\zeta_{27})^+$ for $p=3$ and $F$ is the degree $5$ subfield of $\Q(\zeta_{25})$ for $p=5$.  We compute with spaces of Hilbert modular forms via the Jacquet-Langlands correspondence, which allows us to locate systems of Hecke eigenvalues in the (degree $1$) cohomology of a Shimura curve.  We combine methods of the first author \cite{Dembele2} with those of the second and third authors \cite{GreenbergVoight} to compute with these spaces explicitly.

In the Table below, we list the fields $K$ we construct, the prime $p$ at which $K$ is ramified, the level $\frakN$ of the corresponding Hilbert modular form, the Galois group $\Gal(K/\Q)$, and an upper bound on the root discriminant $\delta_K$ of $K$.  We denote by $\frakp_3$ and $\frakp_5$ the unique primes above $3$ and $5$ in the field $F$.  

\begin{table}[h]
\begin{equation} \notag
\begin{array}{c|c|c|c|c}
p & \frakN & \Gal(K/\Q) & \approx [K:\Q] & \delta_K \leq \\
\hline
3\rule{0pt}{2.5ex} & 1 & 9\cdot \PGL_2(\F_{3^{27}}) & 4.0\cdot 10^{39} & 76.21 \\
  & \frakp_3 & 9\cdot \PGL_2(\F_{3^{18}}) & 5.2 \cdot 10^{26} & 86.10 \\
  & \frakp_3 & 9\cdot \PGL_2(\F_{3^{36}}) & 3.0 \cdot 10^{52} & 86.10 \\
\hline
5\rule{0pt}{2.5ex} & \frakp_5 & 5\cdot \PGL_2(\F_{5^5}) & 1.5 \cdot 10^{12} & 135.39 \\
  & \frakp_5 & 10\cdot \PSL_2(\F_5)^5 & 3600\ (*) & 135.39 \\
  & \frakp_5 & 5\cdot \PGL_2(\F_{5^{10}}) & 4.6 \cdot 10^{20} & 135.48 \\ 
\end{array}
\end{equation}
\begin{center}
\textbf{Table}: Nonsolvable fields $K$ ramified only at $p$
\end{center}
\end{table}

%\todo{Check the Galois group and root discriminant as in Roberts.}

Remarkably, after a preprint of our result was posted, Roberts \cite{Roberts5} exhibited an explicit polynomial in $\Z[x]$ of degree $25$ that generates a field of discriminant $5^{69}$ with Galois group $10 \cdot \PSL_2(\F_5)^5$.  The splitting behavior of primes in its Galois closure matches one of the fields we construct (marked above as $(*)$) as far as we computed.  His construction uses the 5-division field of an elliptic curve defined over the degree 5 subfield of $\Q(\zeta_{25})$. By combining a theorem of Skinner and Wiles~\cite[Theorem 5.1]{SkinnerWiles} and Jarvis~\cite[Mazur's principle]{jarvis}, we prove that in fact the two fields are isomorphic.  Roberts then computes (agreeably) that the root discriminant of this Galois field is equal to $5^{3-1/12500} < 135.39$.

On the assumption of the Generalized Riemann Hypothesis (GRH), the Odlyzko bounds \cite{Martinet} imply that each of the fields we construct is totally imaginary.  Indeed, these bounds \cite{Odlyzko} show that $\delta_K \geq 44$ for totally complex fields $K$ of sufficiently large degree, so these fields come very close (within a factor 2 for $p=3$ and factor $3$ for $p=5$) of these bounds, which is remarkable considering they have such large degree.  As the field that Demb\'el\'e constructs is (provably) totally complex, we leave unanswered the question of whether a totally real nonsolvable extension of $\Q$ exists which is ramified only at $p$ for these primes.

This article is organized as follows.  In Section 1, we review the case $p=2$.  In Sections 2 and 3, we treat in detail the cases $p=3$ and $p=5$, respectively.  Finally, in Section 4, we consider the case $p=7$ and discuss the obstacles we face in applying our methods in this instance.  Throughout, our computations are performed in the \textsf{Magma} computer algebra system \cite{magma1}.

\begin{acknowledgements}
The authors would like to thank Benedict Gross for many inspirational email exchanges and helpful suggestions. They would also like to thank the \textsf{Magma} group at the University of Sydney for their continued hospitality and especially Steve Donnelly for his assistance.  Finally, they are indebted to Richard Foote for the reference \cite{GoLySol} (used in the proof of Case 2 of Lemma \ref{peq5}) and to David Roberts and the referee for their comments and corrections.
\end{acknowledgements}

\section{A review of the case $p=2$}\label{sec:review-p=2}
In this section, we briefly review work of the first author~\cite{Dembele2}, who established Gross' conjecture for $p=2$ as follows.

\begin{thm} \label{thm11}
There exists a nonsolvable number field $K$ which is ramified only at $2$ with Galois group $8\cdot\mathrm{SL}_2(\F_{2^8})^2$.
\end{thm}

Let $F=\Q(\zeta_{32})^+$ be the totally real subfield of the cyclotomic field $\Q(\zeta_{32})$ and let $\Z_F$ denote its ring of integers.  The field $K$ in Theorem \ref{thm11} arises from a Galois representation associated to an eigenform in the space $S_2(1)$ of cuspidal Hilbert modular forms of level $1$ and parallel weight $2$ over $F$.  Let $\T$ denote the Hecke algebra acting on $S_2(1)$, the $\Z$-subalgebra of $\End(S_2(1))$ generated by the Hecke operators $T_\fp$ with $\fp$ a prime of $\Z_F$.  The space $S_2(1)$ has dimension 57 and decomposes into Hecke-irreducible constituents with dimensions $1$, $2$, $2$, $4$, $16$, and $32$. Let $f$ be a newform in the 16-dimensional constituent, and let $\T_f$ denote the restriction of the Hecke algebra to this constituent. (We only consider this constituent because the remaining ones all reduce to zero modulo $2$; see Remark \ref{notinteresting} below for a similar observation when $p=3$.)  Let $H_f=\T_f\otimes\Q=\Q(a_\fp(f))$ be the field of Fourier coefficients of $f$, where $f \slsh{} T_\fp = a_\fp(f)f$, and let $\Delta=\mathrm{Aut}(H_f)$. Let $E$ be the fixed field of $\Delta$ so that $\Gal(H_f/E)=\Delta$. 

By the work of Taylor \cite{taylor1}, there exists a (totally odd) Galois representation
\[ \overline{\rho}_f:\Gal(\overline{F}/F) \to \GL_2(\T_f \otimes \F_2) \]
such that for any prime $\fp \nmid 2$ of $\Z_F$ we have $\tr(\overline{\rho}_f(\Frob_\fp)) \equiv a_\fp(f) \pmod{2}$ and $\det(\overline{\rho}_f(\Frob_\fp)) \equiv N\fp \pmod{2}$. The mod $2$ Hecke algebra $\T_f\otimes\F_2$ has two maximal non-Eisenstein ideals $\mathfrak{m}_f$ and $\mathfrak{m}_f'$ with the same residue field $k=\F_{2^8}$. One obtains the field $K$ as the compositum of the fixed fields of $\ker (\overline{\rho}_f\bmod\mathfrak{m}_f)$ and $\ker (\overline{\rho}_f\bmod\mathfrak{m}_f')$.   It is conjectured, in analogy with the Eichler-Shimura construction for $\Q$,  that there exists a 16-dimensional abelian variety $A_f$ with everywhere good reduction and real multiplication by $H_f$ such that $\overline{\rho}_f$ can be realised in the 2-torsion of $A_f$.

The Galois group $\Gal(F/\Q)$ acts on $\T$ by acting on the (prime) ideals of $\Z_F$. This action clearly preserves the decomposition of $\T$ according to Hecke-irreducible constituents. By Shimura~\cite[Proposition 2.6]{shimura2}, the action of $\Delta$ also preserves those constituents. In particular, both these actions preserve $\T_f$ and hence must be compatible.  Therefore, for each $\sigma\in\Gal(F/\Q)$, there is a unique $\tau=\tau(\sigma)\in\Delta$ such that, for all prime ideals $\fp\subset\Z_F$, we have
\[ a_{\sigma(\fp)}(f)=\tau(a_{\fp}(f)). \]
The map $\tau$ thus yields a homomorphism 
\begin{equation}  \label{Deltatau}
\begin{aligned}
\Gal(F/\Q) &\to \Delta \\
\sigma &\mapsto \tau(\sigma)
\end{aligned}
\end{equation}
which we see in this case is an isomorphism.  

Since $\Gal(F/\Q)$ is abelian, it follows that $H_f$ must be a ray class field over $E$ such that $8=[F:\Q]$ divides $[H_f:\Q]=\dim A_f=16$, if the conjectural abelian variety $A_f$ indeed exists.  By direct calculations, we show that $E=\Q(\sqrt{5})$ and that $H_f$ is the ray class field of conductor $\fP_5\fP_{89}\fP_{661}$, where $\fP_5$ is the unique prime above $5$, and $\fP_{89}$ (resp.\ $\fP_{661}$) is one of the split primes above $89$ (resp.\ $661$).  The ideal $2\Z_{H_f}$ factors as $2\Z_{H_f}=\mathfrak{P}\mathfrak{P}'$, with $\tau(\mathfrak{P})=\mathfrak{P}'$ for any cyclic generator $\tau$ of $\Delta$. (It is not hard to see that the primes $\mathfrak{P}$ and $\mathfrak{P}'$ correspond to the maximal ideals $\mathfrak{m}_f$ and $\mathfrak{m}_f'$ under the reduction map $\T_f\to\T_f\otimes\F_2$.)  

We are grateful to Gross for these beautiful observations, which extend to the cases we consider in the following sections.

\begin{rmk}
One striking fact about this construction is that the field $K$ has a very small root discriminant compared to its degree. Indeed, Serre~\cite{Serre6} shows that the root discriminant $\delta_K$ of $K$ satisfies $\delta_K\le 55.4$.  Although this estimate is much lower than the bounds we obtain for the fields constructed with $p=3$ and $p=5$, it nevertheless suggests that one might be able to construct towers of number fields with low root discriminant bounds in this way.  The smallest such bound for a tower currently known is $82.2$, obtained by Hajir and Maire~\cite{HajirMaire2, HajirMaire1}.  
\end{rmk}

\section{A nonsolvable Galois number field ramified only at $3$}

In this section, we prove that Gross' conjecture is true for $p=3$.  We refer to the work of the second and third author \cite{GreenbergVoight} as a reference for the method employed.

To begin, we choose a totally real base field $F$ which is ramified only at $p=3$.  The enumeration of such fields with small degree has been extensively studied by Jones \cite{JonesTables}, Jones and Roberts \cite{JonesRoberts}, Brueggeman \cite{BrueggemanSeptic}, Lesseni \cite{LesseniOctic, LesseniNonic}, and others.  Ordering fields by degree and then discriminant, we find up to degree $9$ that the only such fields $F \neq \Q$ are the (totally real) subfields of $\Q(\zeta_{27})$, namely $\Q(\zeta_9)^+$ and $\Q(\zeta_{27})^+$ having degrees $3$ and $9$, respectively.

\subsection*{Degree 3}

First consider the field $F=\Q(\zeta_9)^+$ of degree $3$, and let $\Z_F$ denote its ring of integers.  By the Jacquet-Langlands correspondence (see e.g.\ Hida \cite[Proposition 2.12]{HidaCM}), the space of Hilbert modular cusp forms over $F$ is isomorphic as a Hecke module to the space of quaternionic modular forms over the quaternion algebra $B$ which is ramified at $2$ of the $3$ real places of $F$ and no finite place.  Let $X(1)$ be the Shimura curve associated to a maximal order in $B$.  It is well-known that $X(1)$ has genus zero---in fact, it corresponds to the $(2,3,9)$-triangle group \cite[Chapter 5]{VoightThesis}---and consequently there are no Hilbert cusp forms of level $1$ associated to this field.  Raising the level at $3$, the first form occurs in level $3\Z_F=\fp_3^3$, where $\fp_3$ is the prime in $\Z_F$ of norm $3$.  By work of Buzzard, Diamond and Jarvis~\cite[Proposition 4.13(a)]{BDJ}, it then follows that there are only reducible forms in levels that are higher powers of $\fp_3$. 
Indeed, in level $3\Z_F$, the corresponding Shimura curve $X_0(3)$ has genus $1$ and is defined over $\Q$, and Elkies \cite{Elkies} has shown that it is isomorphic to the elliptic curve $A$ defined by $y^2=x^3-432$ (or more symmetrically by $x^3+y^3=1$) with $j(A)=0$.  However, since $A$ has complex multiplication, the mod $3$ representation associated to $E$ is reducible: indeed $A[3](\Q)=A[3](F) \cong \Z/3\Z$.  (See also Brueggeman \cite{BrueggemanNonsolvable} for an argument with Odlyzko bounds which proves, assuming the GRH, that there are only finitely many nonsolvable finite extensions of $F$ with Galois group contained in $\GL_2(\overline{\F}_3)$.)

\subsection*{Degree 9, level 1}

Next, we are led to consider instead the field $F=\Q(\zeta_{27})^+$; here we will find our nonsolvable extension.  Let $\Z_F$ again denote the ring of integers of $F$.  The field $F$ is a cyclic extension of $\Q$ of degree $9$ with discriminant $d_F=3^{22}$, and $F$ is generated by $\lambda=2\cos(2\pi/27)=\zeta_{27}+1/\zeta_{27}$, an element which satisfies
\[ \lambda^9 - 9\lambda^7 + 27\lambda^5 - 30\lambda^3 + 9\lambda - 1 = 0. \]
Moreover, the field $F$ has strict class number $1$.

Let $B$ be the quaternion algebra $B=\quat{-1,u}{F}$ where 
\[ u=-\lambda^6 + \lambda^5 + 5\lambda^4 - 4\lambda^3 - 6\lambda^2 + 3\lambda, \]
i.e.,\ $B$ is generated by $\alpha,\beta$ as an $F$-algebra subject to $\alpha^2=-1$, $\beta^2=u$ and $\beta\alpha=-\alpha\beta$.  The algebra $B$ is ramified at $8$ of the $9$ real places of $F$ and at no finite place.  Indeed, the element $u \in \Z_F^*$ is a unit with the property that $v(u)<0$ for a unique real place $v$ of $F$.  The order $\cO \subset B$ generated as a $\Z_F$-algebra by $\alpha$ and 
\[ \frac{1}{2}\left((\lambda^8 + \lambda^6 + \lambda^4 + \lambda^3 + 1) + (\lambda^8 + \lambda^7 + \lambda^2 + \lambda + 1)\alpha + \beta\right) \]
is a maximal order of $B$.  Let $\Gamma=\Gamma(1)$ denote the Fuchsian group associated to the order $\cO$ and let $X=X(1) = \Gamma \backslash \cH$ be the associated Shimura curve.  

We compute \cite{VoightShim} that the signature of the group $\Gamma$ is $(45; 2^{19}, 3^{13}, 9, 27)$, that is to say, $X$ is a curve of genus $45$ and $\Gamma$ has $19,13,1,1$ elliptic cycles of orders $2,3,9,27$, respectively.  The hyperbolic area of $X$ is equal to $5833/54$.

Next, we compute \cite{V-fd} a fundamental domain for the action of $\Gamma$ by computing a Dirichlet domain $D$ centered at $2i \in \cH$.  We exhibit the domain $D$ inside the unit disc $\cD$ in Figure 2.1, mapping $2i \mapsto 0 \in \cD$.  The domain $D$ has $630$ sides (counted according to convention), and this calculation took a CPU week.  The computation of the domain $D$ yields an explicit presentation of the group, verifies that $X$ has genus $45$, and gives representatives for the $19+13+1+1=34$ elliptic cycles.  

\begin{figure}[h]
\begin{equation} \label{funddom} \notag
\includegraphics{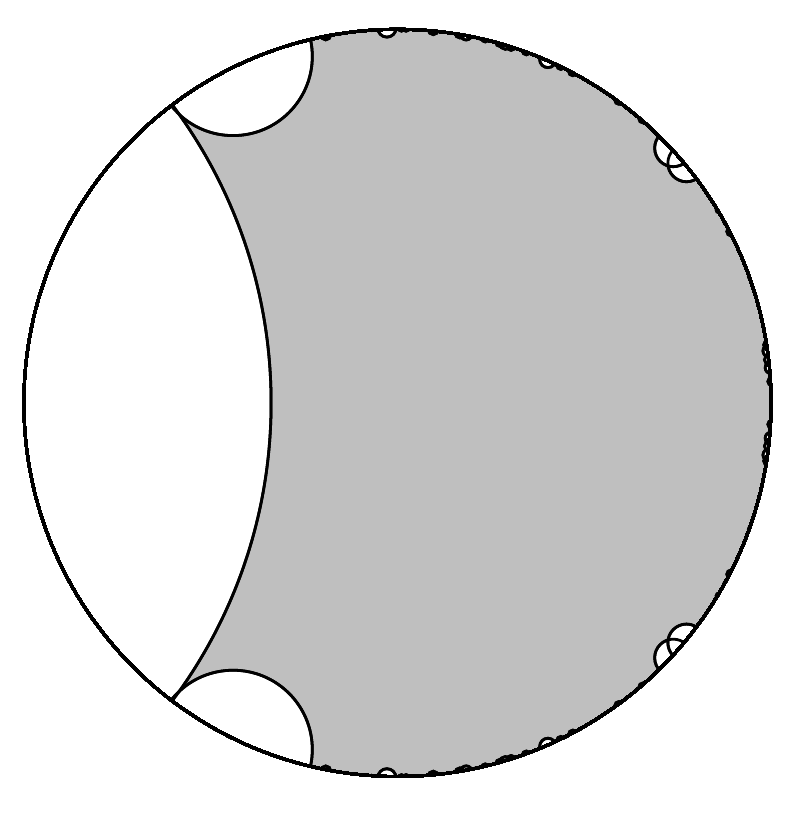}
\end{equation}
\begin{center}
\textbf{Figure \ref{funddom}}: A fundamental domain for the Shimura curve $X(1)$ over $F=\Q(\zeta_{27})^+$
\end{center}
\end{figure}
\addtocounter{equation}{1}

We then compute \cite{GreenbergVoight} the finite set of normalized Hecke eigenforms in the space $S_2(1)$, the space of Hilbert cusp forms associated to $F$ or equivalently the space of quaternionic modular forms associated to $\calO$, by computing normalized Hecke eigenforms in the cohomology group $H^1(\Gamma,\Z)^+$.  We find three Hecke-irreducible components of $S_2(1)$ with dimensions $3$, $6$, and $36$.  

Let $\T$ denote the Hecke algebra acting on $S_2(1)$, the $\Z$-subalgebra of $\End(S_2(1))$ generated by the Hecke operators $T_\fp$ with $\fp$ a prime of $\Z_F$.  Let $f$ be a newform which represents the equivalence class of the third component of $S_2(1)$, having dimension $36$.  (We will see in Remark \ref{notinteresting} below that the other two components do not give rise to interesting Galois representations modulo $3$.)  Let $H_f=\Q(a_\fp(f))$ be the field of coefficients of $f$, where where $f \slsh{} T_\fp = a_\fp(f)f$.  Then, by theory of Eichler and Shimura~\cite{shimura1} (see also Knapp~\cite[Chap. XII]{knapp1}), there is an abelian variety $A_f$ associated to $f$ of dimension $36$ defined over $F$ which is a quotient of the Jacobian $J(1)$ of $X(1)$ with the additional properties that $A_f$ has real multiplication by $H_f$ and has everywhere good reduction.  

% For $i=1,2,3$, let $K_i=\Q(a_\fp(f_i))$.  Direct computations allow us to identify the fields $K_i$.  We first find that $K_1$ is the (totally real) cubic subfield of $\Q(\zeta_{19})$ with discriminant $d_{K_1}=19^2$ and hence $3$ is inert in $K_1$.  The field $K_2$ contains $E=\Q(\zeta_9)^+=\Q(\lambda_9)$, where $\lambda_9^3-3\lambda-1$.  The ideal $\frakf$ generated by $8\lambda^2 - 8\lambda - 3$ has norm $N\frakf=213=3\cdot 71$ (note that $71$ splits completely in $E$) and $K_2$ is the (strict) ray class field of $E$ with conductor $\frakf \infty$.  In particular, $3$ is totally ramified in $K_2$.  
% (For purposes of verification, we note that $\chi(T_{\fp_{53}})$ factors into polynomials of degrees $3,6,36$ which generate number fields of discriminants $19^2$, $3^9 71^1$, and $3^{12} 7^6 71^9 131^9 6481^8$, respectively; the first two are given by $x^3+12x^2-9x-297$ and $x^6+9x^5-108x^4-513x^3+4617x^2-2916x-12393$.)

We identify the field $H_f$ as follows.  Let $E$ be the (totally real) field generated by a root $c$ of the polynomial $x^4 + x^3 - 5x^2 - x + 3$; the field $E$ has discriminant $d_E=9301=71 \cdot 131$ and Galois group $S_4$.  Then $H_f$ is the ray class field of $E$ with conductor 
\[ \fP_3^2\fP_7\fP_{6481} = (1089c^3 + 134c^2 - 5551c + 3747), \] 
where $\fP_3$ and $\fP_{6481}$ are the unique primes in $\Z_E$ of norm $3$ and $6481$, respectively, and $\fP_7=(10c^3+c^2-51c+35)\Z_E$ is one of two primes of norm $7$.  We have $[H_f:E]=9$ and the discriminant of $H_f$ is computed to be $d_{H_f}=3^{12} 7^6 71^9 131^9 6481^8$.  By a ray class group computation, we find that $\Gal(H_f/E) \cong \Z/9\Z$.  Indeed, the homomorphism $\Gal(F/\Q) \to \Gal(H_f/E)$ defined as in (\ref{Deltatau}) is again an isomorphism, which verifies that $H_f$ is a ray class field and that $9=[F:\Q]$ divides $[H_f:\Q]=\dim A_f=36$.

Let $\T_f$ be the (restriction) of the Hecke algebra acting on the constituent corresponding to $f$; then $\T_f$ is an order of conductor $\fP_3'^6$ in the ring of integers $\Z_{H_f}$, where $3\Z_{H_f}=\fP_3\fP_3'^3$ and the primes $\fP_3,\fP_3'$ in $\T_f$ have inertial degrees $27,3$.  In particular, the order $\T_f$ is maximal at the prime $\fP_3$.

Now, by the work of Carayol \cite{carayol1}, there exists a totally odd Galois representation
\[ \rho_f:\Gal(\overline{F}/F) \to \GL_2(\T_f \otimes \Z_3) \]
such that for any prime $\fp \nmid 3$ of $\Z_F$ we have $\tr(\rho_f(\Frob_\fp))=a_\fp(f)$ and $\det(\rho_f(\Frob_\fp))=N\fp$.  In this case, where the degree of the basis field $F$ is odd, the representation $\rho_f$ can be realized explicitly in the $3$-adic Tate module of the abelian variety $A_f$.  %(It is still a conjecture that the $2$-adic representations that appear in work of the first author \cite{Dembele2} can be realized in the same way.)

\begin{rmk} \label{notinteresting}
In our analysis, we consider only the prime $\fP_3$ of inertial degree $27$ in the constituent of $f$ of dimension $36$.  For the other primes above $3$ in $\T \otimes \Z_3$, in all cases the Hecke algebra is not maximal at these primes (there are unique primes above $3$ for the components of dimensions $3$ and $6$, respectively, and one other prime above $3$ for the component of dimension $36$).  In other words, for every such prime $\fP \neq \fP_3$ of $\T \otimes \Z_3$ with $\fP \mid 3$ we have $a_\frakp(f) \equiv 0 \pmod{\fP}$ for all primes $\frakp \nmid 3$, and so the reduction of $\rho_f$ modulo $\fP$ is trivial.
\end{rmk}

We are now ready to prove the following theorem.

\begin{thm} \label{peq3}
There exists a Galois extension $K$ of $\Q$ that is ramified only at $3$ with Galois group $\Gal(K/\Q) \cong 9 \cdot \PGL_2(\F_{3^{27}})$.  
\end{thm}

\begin{proof}
% For the unique prime $\fp_3 \mid 3$ of $\Z_F$, we compute that the characteristic polynomial $\chi(T_{\fp_3})$ of the Hecke operator $T_{\fp_3}$ acting on the Hecke span of $f$ satisfies 
% \begin{equation} \label{Tp3}
% \chi(T_{\fp_3}) \equiv x^9(x^3+x^2+x-1)^9 \pmod{3}.
% \end{equation}

We reduce the representation $\rho_f$ modulo the prime $\fP_3$ of $\T_f$ of degree $27$ to obtain a representation
\[ \overline{\rho}_f:\Gal(\overline{F}/F) \to \GL_2(k) \]
where $k=\F_{3^{27}}$.  Let $c$ satisfy $c^{27}-c^7+1=0$ so that $k=\F_3(c)$ and $k^\times$ is generated by $c$, and let $q=\#k=3^{27}$.  Note in particular that $\ker{\overline{\rho}_f}$ is Galois over $\Q$ since $\Gal(F/\Q)$ fixes the ideal $\fP_3$.

In Table \ref{tableF27apop}, we tabulate the discrete logarithm to the base $c$ of some eigenvalues $\overline{a}_\fp(f)=\tr(\overline{\rho}_f(\Frob_{\fp})) \in k$ occurring in the representation $\overline{\rho}_f$; the logarithm takes values in $\Z/(q-1)\Z$.

\begin{table}[h]
\begin{equation} \label{tableF27apop} \notag
\begin{array}{c||c|c}
N \fp & \log_c \overline{a}_\fp(f) & o_\fp(f) \\
\hline
3\rule{0pt}{2.5ex} & 5279259797298 & - \\
53 & 4309388243332 & q-1 \\
107 & 3543848555542 & q-1 \\
109 & 5965238429265 & (q-1)/2 \\
163 & 3456998555640 & (q-1)/2 \\
269 & 2951025230806 & (q-1)/109 \\
271 & 2766037528324 & (q+1)/4
\end{array}
\end{equation}
\begin{center}
\textbf{Table \ref{tableF27apop}}: Hecke data for the representation $\overline{\rho}_f$ of level $1$ for $F=\Q(\zeta_{27})^+$
\end{center}
\end{table}
\addtocounter{equation}{1}

For example, we have
\[ \overline{a}_{\fp_{53}}(f) = c^{4309388243332} = -c^{25}+c^{24}+c^{23}-c^{21}-c^{18}+c^{15}+c^{13}+c^{11}-c^{10}-c^9-c^8-c^6-c^5-c^4+c^2 \]
and $\overline{a}_{\fp_{53}}(f)$ satisfies the polynomial
\[ x^{27} -x^{24} -x^{23} + x^{22} - x^{20} + x^{19} + x^{18} - x^{17} - x^{16} - x^{15} - x^{14} - x^{13} + x^{11} - x^6 + x^5 + x^4 + x^3 + x - 1. \]

In Table \ref{tableF27apop}, we list values for one choice of prime $\fp$ above $N\fp$; the full set of Hecke eigenvalues for the primes $\fp$ with $\fp \nmid 3$ are simply $\sigma^i(\overline{a}_\fp(f))$ for $i=0,\dots,8$, where $\sigma$ is the $27$-power Frobenius $c \mapsto c^{27}$ on $k$.  

We consider the projection $\rmP\overline{\rho}_f$ of $\overline{\rho}_f$ to $\PGL_2(k)$ and prove that the image is surjective.  For a prime $\frakp \mid 3$ of $\Z_F$, the element $\overline{\rho}_f(\Frob_{\fp})$ has characteristic polynomial $x^2-\overline{a}_\fp(f) x + N\fp$ modulo $3$, and consequently we can compute the order $o_\fp(f)$ of $\overline{\rho}_f(\Frob_{\fp})$ in $\PGL_2(k)$ for the primes in Table \ref{tableF27apop}.  We note that these orders are independent of the choice of prime $\fp$ above $N\fp$.  

With these orders in hand, we refer to Dickson's classification \cite{Dickson} of subgroups of $\PGL_2(k)$.  The image of $\rmP\overline{\rho}_f$ is obviously not an exceptional group; it is not cyclic or dihedral; it is not affine since it contains elements dividing both $q-1$ and $q+1$; and it cannot be contained in $\PSL_2(k)$ since $\overline{\rho}_f$ is totally odd and $-1$ is not a square in $k$.  Therefore the image $\rmP\overline{\rho}_f$ is projective and of the form $\PGL_2(k')$ with $k' \subseteq k$ a subfield, but already the fact that $\lcm(q-1,(q+1)/4) \mid \#\PGL_2(k')$ requires that $k'=k$.  

The desired field $K$ is then simply the Galois extension of $F$ that is invariant under $\ker(\rmP \overline{\rho}_f)$, considered as an extension of $\Q$.
\end{proof}

We observed earlier that the residual representation $\overline{\rho}_f$ can be realized in the $3$-torsion of the abelian variety $A_f$, which has good reduction everywhere.  Thus we can apply a result of Fontaine~\cite[Corollaire 3.3.2]{Fontaine1} to bound the root discriminant of $K$ as 
\[ \delta_K\le\delta_F\cdot 3^{1+\frac{1}{3-1}}=3^{22/9+3/2} \leq 76.21. \]

\subsection*{Degree 9, level $\fp_3$}

We conclude this section with results of computations for the field $F=\Q(\zeta_{27})^+$ with level $\fp_3$.  We compute the space of Hilbert modular cusp newforms $S_2(\fp_3)\spnew$ of level $\fp_3$ in two ways.  As above, we can compute the Hecke module $S_2(\fp_3)\spnew$ in the cohomology of the Shimura curve $X_0(\fp_3)$.  On the other hand, by the Jacquet-Langlands correspondence, we can compute this space \cite{Dembele2} as the space of cuspidal automorphic forms associated to a maximal order of the totally definite quaternion algebra over $F$ which is ramified at $\fp_3$ and all $9$ real places of $F$---and in this case, the computations were performed by Steve Donnelly.  

Agreeably, we find in each case that the space $S_2(\fp_3)\spnew$ has dimension $117$ and decomposes into irreducible Hecke constituents of dimensions 53 and 64.  Due to high complexity we do not perform the same detailed analysis of the coefficient field for these spaces as we did before.  Nevertheless, we again obtain a Galois representation 
\[ \overline{\rho}:\Gal(\overline{F}/F)\to\GL_2(\T\otimes\F_3). \]
The $\F_3$-algebra $\T\otimes\F_3$ has 12 non-Eisenstein maximal ideals. In Table \ref{tableF27}, we group these maximal ideals according to the degree of their residue fields and note the action of $\Gal(F/\Q)$, generated by $\sigma$.

\begin{table}[h]
\begin{equation} \label{tableF27} \notag
\begin{array}{c||c|c}
\text{Inertial degree}&\text{Number of ideals}&\text{Galois action}\\ \hline
1\rule{0pt}{2.5ex}&9&\text{$\sigma$ permutes} \\
18&2&\text{$\sigma$ fixes} \\
36&1& \text{$\sigma$ fixes} \\
\end{array}
\end{equation}
\begin{center}
\textbf{Table \ref{tableF27}}: Hecke data for the representation $\overline{\rho}$ of level $\fp_3$ for $F=\Q(\zeta_{27})^+$
\end{center}
\end{table}
\addtocounter{equation}{1}

Arguing as above, we find four other Galois extensions ramified only at $3$: one with (solvable) Galois group $18 \cdot \PSL_2(\F_3)^9$, two with Galois group $9\cdot \PGL_2(\F_{3^{18}})$, and one with Galois group $9\cdot \PGL_2(\F_{3^{36}})$.
In order to estimate the root discriminants of these fields, we note that by consideration of inertial degrees, there is no congruence modulo $3$ between the forms computed here and the forms of level $1$. Therefore, the restrictions to the decomposition group at $\fp_3$ of the corresponding residual Galois representations are \emph{tr\`es ramifi\'ee} and we cannot directly apply the result of Fontaine as above (though his remark \cite[Remarque 3.3.3]{Fontaine1} implies that one should be able to adapt the argument). We use instead a result of Moon~\cite[Lemma 2.1]{Moon1} to obtain that the root discriminants are bounded by $3^{45/18+(1+10/18)\left(1-\frac{1}{3^m}\right)}$, with $m=9,18,36$, respectively; thus they are all bounded by $3^{73/18} \approx 86.098$.  We provide the details of this argument in the next section (Remark \ref{disccalc}) as we may then directly compare it with Roberts' calculation.

\section{Nonsolvable Galois number fields ramified only at $5$}\label{section-p5}

In this section, we prove that Gross' conjecture is true for $p=5$.  Our method follows the same lines of reasoning as in the previous section.

As above (see also Brueggeman \cite{brueggeman1}, Lesseni \cite{LesseniDecic}), all candidates (ramified only at $5$) for the base field $F$ up to degree $10$ are subfields of $\Q(\zeta_{25})^+$.  We find only reducible representations arising from the field $\Q(\sqrt{5})=\Q(\zeta_5)^+$; see work of the first author \cite{Dembelesqrt5} for a detailed analysis of Hilbert modular forms over this quadratic field.

\subsection*{Degree 5, level 1}

Let $F$ be the degree $5$ subfield of $\Q(\zeta_{25})^+$.  Then $F$ is a cyclic extension of $\Q$ with discriminant $d_F=5^8$, and $F=\Q(b)$ where $b^5-10b^3-5b^2+10b-1=0$.  The field $F$ has strict class number $1$.
% Let $E$ be the maximal totally real subfield of $\Q(\zeta_{25})$ so that $E=\Q(\beta)$, where $\beta=\zeta_{25}+1/\zeta_{25}$. Let $F$ be the degree 5 cyclic subfield of $E$. Then $F=\Q(\beta')$, where $\beta':=\beta^8 - 7\beta^6 + 14\beta^4 - 7\beta^2$. Let $\sigma\in\mathrm{Gal}(F/\Q)$ be the cyclic generator given by $(\beta'\mapsto \beta^7 - 7\beta^5 + 14\beta^3 - 6\beta)$. %$(\beta'\mapsto 1/7(-3\beta'^4 - 2\beta'^3 + 24\beta'^2 - 6\beta' - 6))$.

We proceed as in Section 2.  We first find a unit $u \in \Z_F^*$
% u=(-38b^4 - 5b^3 + 381b^2 + 237b - 363)/7 \in \Z_F^*$ 
such that the algebra $B=\quat{-1,u}{F}$ is ramified at all but one real place of $F$ and no finite place.  Next, we let $\calO$ be a maximal order in $B$, and we let $\Gamma(1)$ and $X(1)$ denote the Fuchsian group and Shimura curve associated to $\calO$, respectively.  The signature of $\Gamma(1)$ is $(2;2^5,3^{11})$ and the area of $X(1)$ is $71/6$.  Here, the field $F$ and the curve $X(1)$ are much simpler than the case considered in the previous section and our computations are significantly less laborious.  We find that the space $S_2(1)$ of Hilbert cusp forms of level $1$ is irreducible (of dimension $2$), corresponding to a normalized cusp form $f$.  We compute the eigenvalues $a_\fp(f)$ of the Hecke operators $T_\fp$ in Table \ref{tableF5l1}; here we let $\omega$ satisfy $\omega^2+\omega-1=0$ so that $\T=\Z[\omega]$ is the ring of integers in $\Q(\sqrt{5})$.  

\begin{table}[h]
\begin{equation} \label{tableF5l1} \notag
\begin{array}{c||c}
N \fp & a_\fp(f) \\
\hline
5\rule{0pt}{2.5ex} & \omega-2 \\
7 & \omega \\
32 & -5\omega+2 \\
43 & -3\omega-3 \\
101 & 4\omega-1 \\
107 & -12\omega-9 \\
149 & -8\omega+1 \\
151 & 9\omega+9 \\
157 & 10\omega+7 \\
193 & 6\omega+4 \\
199 & -9\omega-12 \\
243 & 31
\end{array}
\end{equation}
\begin{center}
\textbf{Table \ref{tableF5l1}}: Hecke eigenvalues in level $1$ for $F \subset \Q(\zeta_{25})^+$ with $[F:\Q]=5$ 
\end{center}
\end{table}
\addtocounter{equation}{1}

We list only the norm of the prime $\fp$ in Table \ref{tableF5l1} since the Hecke eigenvalue $a_\fp(f)$ is independent of the choice of prime with given norm for all such primes $\fp$.  This observation suggests that the form $f$ is the base change of a form from $\Q$.  Indeed, because $F$ has strict class number $1$, the genus $2$ curve $X(1)$ has a canonical model defined over $F$, and because the data defining $B$ is $\Gal(F/\Q)$-invariant, by functoriality (Galois descent) the curve $X(1)$ in fact has field of moduli equal to $\Q$.  But then since $[F:\Q]$ is odd, and $F$ is a field of definition for $X(1)$, the Brauer obstruction to defining $X(1)$ over $\Q$ vanishes (by work of Mestre, completed by Cardona and Quer \cite{CardonaQuer}), so $X(1)$ is defined over $\Q$.  

In fact, we can identify this base change explicitly: let $g$ be the newform of level $\Gamma_1(25)$ and weight $2$ with $\Q$-expansion given by
\[ g = q + (\zeta_5^3-\zeta_5-1)q^2 + \ldots + (\zeta_5^3+\zeta_5^2)q^7 + \ldots + (-3\zeta_5^3-3\zeta_5^2-3)q^{43} + \ldots, \] and $\chi:\,(\Z/25\Z)^\times\to\C^\times$ its associated character. Then $\chi$ has conductor 25 and order 5; and the form $g$ and its conjugates generate $S_2(25,\chi)\spnew$.

\begin{lem}\label{BC} The form $f$ is the base change of $g$ to $F$.
\end{lem}
\begin{proof} Let $BC_{E/F}(f)$ be the base change of $f$ from $F$ to $E=\Q(\zeta_{25})^+$, and let $BC_{F/\Q}(g)$ (resp.\ $BC_{E/\Q}(g)$) the base change of $g$ from $\Q$ to $F$ (resp.\ from $\Q$ to $E$). Then, since $f$ is a newform of parallel weight 2 and level 1, so is $BC_{E/F}(f)$. 

Let $A_g$ be the modular abelian variety associated to $g$ by the Eichler-Shimura construction.
By Mazur and Wiles \cite[Chapter 3.2, Proposition 2]{MazurWiles} (see also Schoof \cite[Theorem 1.1]{Schoof} and the discussion thereafter), $A_g$ acquires everywhere good reduction over $E$. Hence, $BC_{E/\Q}(g)$ is also a newform of parallel weight 2 and level 1.  By computing the space of Hilbert modular forms over $E$ of level $1$ (see the Degree 10 case later in this section for the details of this calculation), we conclude that
$$BC_{E/F}(BC_{F/\Q}(g))=BC_{E/\Q}(g)=BC_{E/F}(f).$$ Hence, by functoriality of base change (see G\'eradin and Labesse \cite[Theorem 2]{GeradinLabesse}), we must have $BC_{F/\Q}(g)=f\otimes \eta$ for some character $\eta:\,\Gal(E/F)\to\{\pm 1\}$. But we see that $a_\fp(BC_{F/\Q}(g))=a_\fp(f)$ for any prime $\fp\mid 7$ in $\Z_F$. Thus $\eta = 1$, since all the primes above $7$ in $\Z_F$ are inert in $\Z_E$, and $BC_{F/\Q}(g)=f$ as claimed. 
\end{proof}
Let
\[ \overline{\rho}_g:\Gal(\overline{\Q}/\Q) \to \GL_2(\F_5) \]
be the mod $5$ representation associated to $g$.  By the proof of Serre's conjecture by Khare \cite{khare1}, the image of $\overline{\rho}_g$ is solvable and by Lemma~\ref{BC}, we have $\overline{\rho}_f=\overline{\rho}_g|_{\Gal(\overline{F}/F)}$.  Therefore, the image of $\overline{\rho}_f$ is solvable as well.  For completeness, we identify the field cut out by $\overline{\rho}_f$.
\begin{lem}
The fields cut out by $\overline{\rho}_g$ and $\overline{\rho}_f$ are $\Q(\zeta_5)$ and $F(\zeta_5)$, respectively.
\end{lem}
\begin{proof}
Computing Hecke eigenvalues of $g$ we confirm that 
\[ pa_p(g) \equiv p^3(p+1) \pmod{5} \]
for all primes $p$, thus $\theta g \equiv \theta^3 G_2 \pmod{5}$ where $G_2$ is the weight $2$ Eisenstein series and $\theta$ is the usual theta operator on modular forms modulo $p$.  It follows from this congruence that $\overline{\rho}_g$ is the direct sum of two powers of the mod $5$ cyclotomic character and, thus,  that $\overline{\rho}_g$ cuts out $\Q(\zeta_5)$.  (Alternatively, this follows from Ellenberg~\cite[Proposition 1.2]{Ellenberg}.)  Therefore, $\overline{\rho}_f=\overline{\rho}_g|_{\Gal(\overline{F}/F)}$ cuts out $F(\zeta_5)$.
\end{proof}

\begin{comment}
(Note that each of the polynomials $\chi(T_\fp)$ listed in Table 2.1 is a square modulo $5$, so in fact the form $f$ is congruent to its Galois conjugate modulo $5$.)  We list in Table 2.1 the eigenvalues $\overline{a}_\fp(f)$ of $\overline{\rho}_f$.  

We now show that the image of $\overline{\rho}_f$ is in fact solvable.  Note that a prime $p \neq 5$ splits (completely) in $F$ if and only if $p \equiv 1,7,18,24 \pmod{25}$.  We observe that for primes $\frakp \nmid 5$ of degree $1$ in $\Z_F$ with $N\frakp < 250$ we have that
\[
a_\fp = \begin{cases}
0, &\text{ if $N\fp \equiv 24 \pmod{25}$}; \\ 
1, &\text{ if $N\fp \equiv 18 \pmod{25}$}; \\
2, &\text{ if $N\fp \equiv 1,7 \pmod{25}$}.
\end{cases} \]
In fact, we find that 
\[ a_\fp \equiv \displaystyle{\legen{N\fp}{5}} + N\fp^{-1} \equiv N\fp^{2} + N\fp^{-1} \pmod{5} \]
and that the field $K$ cut out by $\ker \overline{\rho}_f$ is simply $K=F(\zeta_5)=\Q(\zeta_{25})$.
\end{comment}

\subsection*{Degree 5, level $\fp_5$}

To find a Galois representation of $F$ which has nonsolvable image, we now raise the level.  Let $\cO_0(\fp_5) \subset B$ be an Eichler order of level $\fp_5$ with $\fp_5$ the unique prime above $5$, and let $X_0(\fp_5)$ be the Shimura curve of level $\fp_5$ associated to $\cO_0(\fp_5)$.  Then $X_0(\fp_5)$ is defined over $F$ and has genus $34$.  

We compute the space of Hilbert modular cusp newforms $S_2(\fp_5)\spnew$ of level $\fp_5$ again in two ways: the Hecke module $S_2(\fp_5)\spnew$ occurs both in the cohomology of the Shimura curve $X_0(\fp_5)$ and in the space of cuspidal automorphic forms associated to the maximal order of a totally definite quaternion algebra over $F$ which is ramified at $\fp_5$ and all real places of $F$.  Agreeably, we find in each case that the space has dimension $30$.  

Let $\T\spnew$ be the Hecke algebra acting on $S_2(\fp_5)\spnew$, the $\Z$-subalgebra of $\T$ generated by all Hecke operators $T_\fp$ with $\fp \neq \fp_5$.  The space $S_2(\fp_5)\spnew$ has two irreducible components under the action of $\T\spnew$, of dimensions $10$ and $20$, corresponding to newforms $f$ and $g$.  Let $H_f$ and $H_g$ denote the field of coefficients of $f$ and $g$, respectively.  Again by the theory of Eichler and Shimura, we can decompose the new part of the Jacobian $J_0(\fp_5)\spnew$ of $X_0(\fp_5)$ up to isogeny as
\[ J_0(\fp_5)\spnew \xrightarrow{\sim} A_f \times A_g \]
where $A_f$ and $A_g$ are abelian varieties of dimension $10$ and $20$ with good reduction outside $\fp_5$ and which have real multiplication by $H_f$ and $H_g$, respectively.  

We now identify the fields $H_f$ and $H_g$ by direct computation.  The field $H_f$ is the ray class field of $E_f=\Q(\sqrt{421})$ with conductor $\fP_{11}$, where $\fP_{11}$ is a prime above $11$ in $\Z_{E_f}$.  We find that $[H_f:E_f]=5$ and $d_{H_f}=11^4 421^5$. The restriction $\T_f$ of the Hecke algebra $\T\spnew$ acting on the constituent of $f$ is then identified with the maximal order $\Z_{H_f}$ of $H_f$.  The ideal $5$ splits in $E_f$ and each of these primes remain inert in $H_f$, so that
\begin{equation} \label{5split}
5\T_f = \fP_5 \fP_5'
\end{equation}
with the inertial degree of the primes $\fP_5,\fP_5'$ being equal to $5$.  In particular, we note that $\Aut(H_f)= \Gal(H_f/E_f) \cong \Z/5\Z$ fixes each of the ideals $\fP_5$ and $\fP_5'$.  

The field $H_g$ is obtained as follows.  Let $E_g$ be the (totally real) quartic field generated by a root of the polynomial $x^4 + 2x^3 - 75x^2 - 112x + 816$.  Then the field $E_g$ has discriminant $d_{E_g}=2^2 5^1 13936963$ and Galois group $S_4$.  The prime $71$ splits completely in $E_g$ and $H_g$ is the ray class field of $E_g$ of conductor $\fP_{71}$, where $\fP_{71}$ is a prime above $71$.  We have $[E_g:H_g]=10$ and $d_{H_g}=2^{10} 5^5 71^4 13936963^5$.  The restriction $\T_g$ of $\T\spnew$ to the constituent of $g$ is a suborder of index $32$ in the maximal order of $H_f$.  Here, we have
\begin{equation} \label{5split3}
5\Z_{E_g} = \fP_5^2 \fP_5';
\end{equation}
the prime $\fP_5$ splits completely in $H_g$ and the primes above $\fP_5$ in $H_g$ are permuted by $\Aut(H_g) = \Gal(H_g/E_g) \cong \Z/5\Z$ whereas the prime $\fP_5'$ is inert.

We note again in each of these cases the map defined in (\ref{Deltatau}) is an isomorphism, accounting for the fact that the fields $H_f$ and $H_g$ are ray class fields and $5=[F:\Q]$ divides both $[H_f:\Q]=\dim A_f=10$ and $[H_g:\Q]=\dim A_g=20$.

As above, corresponding to the newforms $f$ and $g$ are representations
\begin{equation} \label{fandg}
\rho_f,\rho_g: \Gal(\overline{F}/F) \to \GL_2(\T_f \otimes \Z_5), \GL_2(\T_g \otimes \Z_5)
\end{equation}
which arise in the $5$-adic Tate modules of the abelian varieties $A_f$ and $A_g$, respectively.  The reductions of these forms indeed give rise to nonsolvable extensions.

\begin{thm} \label{peq5}
There exist Galois extensions $K,K'$ of $\Q$ that are linearly disjoint over $F$ and ramified only at $5$ with Galois group 
\[ \Gal(K/\Q), \Gal(K'/\Q) \cong 5\cdot\PGL_2(\F_{5^5}). \]
Furthermore, there exist Galois extensions $L$ and $M$ of $\Q$ that are ramified only at $5$ with Galois groups $\Gal(L/\Q) \cong 10 \cdot \PSL_2(\F_5)^5$ and $\Gal(M/\Q) \cong 5 \cdot \PGL_2(\F_{5^{10}})$.
\end{thm}

\begin{proof} 
Our proof falls into two cases, corresponding to the representations $\rho_f$ and $\rho_g$ in (\ref{fandg}).

\begin{proof}[Case $1$]
\renewcommand{\qedsymbol}{}
We first exhibit the extensions $K$ and $K'$.  Let $\fm_f,\fm_f'$ be the maximal ideals of $\T_f\otimes\Z_5$ with residue field $k=\F_{5^5}$ by (\ref{5split}), and let $\overline{\rho}_f$ and $\overline{\rho}_f'$ be the reduction of $\rho_f$ modulo $\fm_f$ and $\fm_f'$, respectively.  

We now tabulate the data for the Hecke operators for the representations $\overline{\rho}_f$ and $\overline{\rho}_f'$.  Let $c \in \overline{\F}_5$ be a root of the polynomial $x^5 - x - 2$, so that $k=\F_5(c)$.  In Table \ref{tableffp}, we list for a prime $\fp$ of $\Z_F$ the eigenvalue $\overline{a}_\fp(f)=\tr(\overline{\rho}_f(\Frob_\fp)) \in k$ occurring in the representation $\overline{\rho}_f$ and the order $o_\fp(f)$ of the image $\rmP\overline{\rho}_f(\Frob_\fp) \in \PGL_2(k)$, and we list the similar quantities for the representation $\overline{\rho}_f'$.  

\begin{table}[h]
\begin{equation} \label{tableffp} \notag
\begin{array}{c||cc|cc}
N \fp & \overline{a}_\fp(f) & \overline{a}_\fp(f') & o_\fp(f) & o_\fp(f') \\
\hline
5\rule{0pt}{2.5ex} & 4 & 4 & - & - \\
7 & 3c^4+3c^3+c^2+2c+2 & 3c^4+c^3+4c+2 & 1042 & 284 \\
32 & 0 & 0 & 2 & 2 \\
43 & 2c^4+3c^3+c^2+4c & 3c^4+c^3+c^2+c+2 & 1042 & 3124 \\
101 & 2c^4+3c^3+c^2+4 & 2c^4+4c^3+4c+1 & 781 & 1562 \\
107 & 4c^4+4c^2+c+4 & 2c^4+4c^3+2c^2 & 3126 & 3124 \\
149 & 2c^4+2c^3+4c^2+2c+2 & 4c^4+4c^3+2c+2 & 1563 & 142 \\
151 & 3c^4+4c^3+4c^2+3c+4 & 2c^4+4c^3+2c^2+3c+2 & 781 & 1563 \\
157 & 2c^4+c^2+2c+1 & 2c^4+c^3+4c^2+3c+3 & 1042 & 284 \\
193 & 3c^4+2c^3+4c^2+2c & 4c^4+c^3+4c^2+2c & 3124 & 3124 \\
199 & 4c^4+2c^3+4c^2+1 & 3c^4+4c^3+c^2+4c+4 & 1562 & 781 \\
243 & 0 & 1 & 2 & 4 \\
257 & 2c^4+4c^3+2c^2+2c & 4c^4+3c^2+c+4 & 3126 & 3126
\end{array}
\end{equation}
\begin{center}
\textbf{Table \ref{tableffp}}: Hecke data for the representations $\overline{\rho}_f$ and $\overline{\rho}_f'$ of level $\fp_5$ for $F \subset \Q(\zeta_{25})^+$ with $[F:\Q]=5$
\end{center}
\end{table}
\addtocounter{equation}{1}

We list in Table \ref{tableffp} the eigenvalues for a choice of prime $\fp$ with specified norm; the others can be recovered from the fact that 
there exists $\sigma\in\Gal(F/\Q)$ such that
\[ \overline{a}_{\sigma(\fp)}(f)=\tau(\overline{a}_\fp(f)) \] 
where $\tau: k \to k$ given by $c \mapsto c^5$ is the $5$-power Frobenius.  

An argument as in the proof of Theorem \ref{peq3} shows that the images $\rmP \overline{\rho}_f$ and $\rmP \overline{\rho}_f'$ are surjective to $\PGL_2(k)$.  Let $K$ and $K'$ be the Galois extensions of $F$ cut out by $\ker(\rmP \overline{\rho}_f)$ and 
$\ker(\rmP \overline{\rho}_f')$, respectively, each with Galois group $5 \cdot \PGL_2(\F_{5^5})$.  Since $\Aut(H_f)=\Gal(H_f/E_f)$ fixes $\fm_f$ and $\fm_f'$ individually, we find that $K$ and $K'$ are Galois over $\Q$, as claimed.
\end{proof}

\begin{proof}[Case $2$]
\renewcommand{\qedsymbol}{}
In a similar fashion, we exhibit the extensions $L$ and $M$ from the form $g$.  Now, the maximal ideals of $\T_g \otimes \Z_5$ are as follows: there are five maximal ideals $\fm_g^{(i)}$ for $i=1,\dots,5$, each with residue field $\F_5$, and one maximal ideal $\fm_g$ with residue field $\F_{5^{10}}$.  As in Case 1, let $\overline{\rho}_g^{(i)}$ and $\overline{\rho}_g$ be the reductions of $\rho_g$ modulo the ideals $\fm_g^{(i)}$ and $\fm_g$, respectively.  In Table \ref{tableggp}, we first list the relevant Hecke data for the representations $\overline{\rho}_g^{(i)}$.

\begin{table}[h] 
\begin{equation} \label{tableggp} \notag
\begin{array}{c||cc|cc|cc|cc|cc}
 & \multicolumn{2}{c|}{(1)} & \multicolumn{2}{c|}{(2)} & \multicolumn{2}{c|}{(3)} & \multicolumn{2}{c|}{(4)} &
 \multicolumn{2}{c}{(5)} \\
N\fp & \overline{a}_{\fp}(g) & o_{\fp}(g) & \overline{a}_{\fp}(g) & o_{\fp}(g) & \overline{a}_{\fp}(g) & o_{\fp}(g) &
\overline{a}_{\fp}(g) & o_{\fp}(g) & \overline{a}_{\fp}(g) & o_{\fp}(g) \\
\hline
5\rule{0pt}{2.5ex} & 1 & - & 1 & - & 1 & - & 1 & - & 1 & - \\
7 & 0 & 2 & 4 & 6 & 4 & 6 & 0 & 2 & 2 & 4\\
32 & 2 & 4 & 2 & 4 & 2 & 4 & 2 & 4 & 2 & 4\\
43 & 3 & 6 & 1 & 4 & 1 & 4 & 1 & 4 & 2 & 4\\
101 & 0 & 2 & 1 & 3 & 0 & 2 & 1 & 3 & 3 & 5\\
107 & 2 & 4 & 3 & 4 & 2 & 4 & 3 & 4 & 1 & 6\\
149 & 1 & 5 & 3 & 3 & 1 & 5 & 1 & 5 & 0 & 2\\
151 & 4 & 3 & 0 & 2 & 3 & 5 & 0 & 2 & 2 & 5\\
157 & 3 & 4 & 0 & 2 & 1 & 6 & 3 & 4 & 1 & 6\\
193 & 0 & 2 & 4 & 4 & 2 & 6 & 3 & 6 & 1 & 4 \\
199 & 1 & 5 & 4 & 5 & 4 & 5 & 1 & 5 & 0 & 2\\
243 & 0 & 2 & 0 & 2 & 0 & 2 & 0 & 2 & 0 & 2\\
\end{array}
\end{equation}
\begin{center}
\textbf{Table \ref{tableggp}}: Hecke data for the representations $\overline{\rho}_g^{(i)}$ of level $\fp_5$ for $F \subset \Q(\zeta_{25})^+$ with $[F:\Q]=5$
\end{center}
\end{table}
\addtocounter{equation}{1}

In the column labeled $(i)$ for $i=1,\dots,5$ we give the data for the prime ideal $\tau^{i-1}(\fp)$ where $\tau$ is a generator of $\Gal(H_g/E_g)$.  We tabulate the data for a choice of prime $\fp$ with specified norm, but we have the relation
\[ \overline{a}_{\fp}(g^{(i)})=\overline{a}_{\fp^{(i)}}(g), \]
i.e., the cyclic group $\Z/5\Z$ simultaneously permutes the representations and the multiset of eigenvalues (and the corresponding orders).

It is clear from the data in Table \ref{tableggp} that the representation $\rmP \overline{\rho}_g^{(i)}$ is surjective for each $i=1,\dots,5$.  Let $L_i$ be the Galois extension of $F$ cut out by $\ker(\rmP \overline{\rho}_g^{(i)})$ for each $i$ and let $L$ be the compositum of the $L_i$.  Since $\tau \in \Aut(H_g)= \Gal(H_g/E_g)$ cyclically permutes the maximal ideals $\fm_g^{(i)}$, the group $\Gal(F/\Q)$ permutes the $L_i$ and therefore $L$ is Galois over $\Q$.  

Now we show that $\Gal(L/\Q) \cong 10 \cdot \PSL_2(\F_5)^5$.  It is enough to show that $\Gal(L/F) \cong 2\cdot\PSL_2(\F_5)^5$.  We have a natural injection
\begin{align*}
j: \Gal(L/F) &\hookrightarrow \prod_{i=1}^5 \PGL_2(\F_5) \\
\sigma &\mapsto (\rmP \overline{\rho}_g^{(i)} \sigma|_{L_i})_{i=1}^5.
\end{align*}
Let $G$ be the image of $j$. Then $G$ is contained in the subgroup $2\cdot\PSL_2(\F_5)^5$, which is the image of
$$\left\{(\gamma_1,\ldots,\gamma_5)\in\GL_2(\F_5)^5: \det(\gamma_1)=\cdots=\det(\gamma_5)\right\}$$ under the projection $\GL_2(\F_5)^5\to\PGL_2(\F_5)^5$. Furthermore, the composition $\mathrm{pr}_i\circ j:\,\Gal(L/F)\to\PGL_2(\F_5)$, where $\mathrm{pr}_i$ is the projection on the $i$th factor, is surjective.
Let $H=j^{-1}(\PSL_2(\F_5)^5)$ and $H'=j|_H(H)\subset\PSL_2(\F_5)^5$ its image. Then $H'$ projects surjectively onto each factor of $\PSL_2(\F_5)^5$ and is stable under cyclic permutation of these factors. Therefore, since $\PSL_2(\F_5)$ is nonabelian and simple, a lemma of Gorenstein, Lyons, and Solomon \cite[Part I, Chapter G, Section B3, Lemma 3.25, p.\ 13]{GoLySol} implies that either $H'\subset\PSL_2(\F_5)^5$ is a  diagonal subgroup or $H'=\PSL_2(\F_5)^5$. However, the former case is impossible for the following reason: the automorphisms of $\PGL_2(\F_5)$ preserve traces, and from Table \ref{tableggp} we see that for a prime $\fp$ with $N\frakp=193$, the elements $\overline{a}_\fp(g^{(i)})$ are pairwise distinct for $i=1,\dots,5$.  It follows that $G=\Gal(L/F) \cong 2\cdot H' = 2 \cdot \PSL_2(\F_5)^5$ as claimed.

Finally, we consider the representation $\overline{\rho}_g:\Gal(\overline{F}/F) \to \GL_2(k)$ where $k=\F_{5^{10}}$.  Let $d \in \overline{\F}_5$ be a root of the polynomial $x^{10} + 3x^5 + 3x^4 + 2x^3 + 4x^2 + x + 2$.  We compile the Hecke data for this case in Table \ref{tableafo}.

\begin{table}[h]
\begin{equation} \label{tableafo} \notag
\begin{array}{c||c|c}
N \fp & \overline{a}_\fp(f) & o_\fp(f) \\
\hline
5\rule{0pt}{2.5ex} & 1 & - \\
7 & d^9+4d^7+2d^6+d^4+d & 4882813 \\
32 & 3d^9+4d^7+d^6+3d^5+3d^4+d^3+2d^2+3d+4 & 13 \\
43 & 4d^9+3d^7+d^6+4d^5+4d^4+3d^2+d+4 & 4882813 \\
101 & d^9+4d^7+3d^6+d^4+3d^3+3d^2+4d+1 & 4882813 \\
107 & 4d^9+d^7+3d^6+4d^4+4d^3+3d^2+d+2 & 1627604 \\
149 & d^6+d^3+2d^2+2d+4 & 4882813 \\
151 & 4d^9+2d^7+4d^6+4d^5+4d^4+2d^2+d & 4882813 \\
157 & d^9+d^6+3d^5+d^4+2d^3+4d^2+2d+1 & 1627604 \\
193 & 2d^9+4d^7+d^6+2d^4+3d^3+3d+1 & 4882813 \\
199 & 3d^9+3d^7+3d^5+3d^4+3d^3+3d^2+d & 4882813 \\
243 & 3 & 6\\
\end{array}
\end{equation}
\begin{center}
\textbf{Table \ref{tableafo}}: Hecke data for the representation $\overline{\rho}_g$ of level $\fp_5$ for $F \subset \Q(\zeta_{25})^+$ with $[F:\Q]=5$
\end{center}
\end{table}
\addtocounter{equation}{1}

Arguments similar to the above show that the representation $\rmP\overline{\rho}_g$ is surjective, and we obtain the field $M$ as the fixed field of $\ker(\rmP\overline{\rho}_g)$.
\end{proof}

Combining these two cases, we have proved the proposition.
\end{proof}

By Shepherd-Barron and Taylor~\cite[Theorem 1.2]{ShepherdTaylor}, there exists an elliptic curve $E/F$ such that the extensions $L_i$ are realized in the 5-torsion of the $\Gal(F/\Q)$-conjugates of $E$. In fact, let 
\begin{multline*}
j=\frac{1}{5^1 7^{11}}(-16863524372777476b^4 - 21086272446873684b^3 + \\ 175687878237839123b^2 + 243736518871536282b    - 27968733551340565).
%j = \frac{1}{17084103059520}(-16863524372777476b^4 -21086272446873684b^3 +\\
 % 175687878237839123b^2 + 243736518871536282b - 27968733551340565),
\end{multline*}
Let $E$ be the elliptic curve over $F$ with $j$-invariant $j$ and minimal conductor (recall that $F$ has class number $1$).  (The curve $E$ was found by Roberts \cite{Roberts5} by a search based on Theorem~\ref{peq5} and the data we gathered in Table~\ref{tableggp}.) Then the conductor of $E$ factors as a product $\mathfrak{p}_5\mathfrak{p}_7$ with $N\mathfrak{p}_5=5$ and $N\mathfrak{p}_7=7$, and the ideal $\mathfrak{p}_7$ is generated by the element $(-4b^4 - 2b^3 + 39b^2 + 50b - 57)/7$.  
By combining work of Skinner and Wiles~\cite[Theorem 5.1]{SkinnerWiles} and a variant of the argument of the first author~\cite[Section 4]{Dembelesqrt5}, it follows that the $3$-adic representation $\rho_{E,3}$, and hence $E$ itself, is modular. Let $X_0(\fp_5\fp_7)$ be the Shimura curve associated to an Eichler order of level $\fp_5\fp_7$ contained in the maximal order $\cO$, and let $J_0(\fp_5\fp_7)$ be the Jacobian of $X_0(\fp_5\fp_7)$.  We compute that $J_0(\fp_5\fp_7)\spnew$ has dimension 203.  Since $E$ is modular and has multiplicative reduction at $\fp_5$, it is a quotient of $J_0(\fp_5\fp_7)\spnew$.  In fact, we find a unique Hilbert newform $f_E$ of parallel weight $2$ and level $\fp_5\fp_7$ with integer Fourier coefficients which therefore must correspond to $E$.

It follows from Roberts' construction~\cite{Roberts5} that the representation $\bar{\rho}_{E,5}$ of $\Gal(\overline{F}/F)$ on the 5-torsion of $E$ is surjective and unramified at $\mathfrak{p}_7$.  Therefore, the level $\mathfrak{p}_5\mathfrak{p}_7$ is a nonoptimal level for $\bar{\rho}_{E,5}$.  Thus, by Jarvis~\cite[Mazur's Principle]{jarvis}, we have $\bar{\rho}_{E,5}\cong\bar{\rho}_g^{(i)}$ for some $1\le i\le 5$.  In other words, the extension constructed by Roberts~\cite{Roberts5} is isomorphic to our field $L$.

\begin{rmk} \label{disccalc}
We estimate the root discriminants of the fields in Theorem~\ref{peq5} as follows. Let $\mathfrak{P}$ be a prime of $\Z_L$ above $\mathfrak{p}_5$. Since $L$ is Galois over $F$, the group $\Gal(L/F)$ acts transitively on the set of those primes and we only need to estimate the discriminant $d_{L_\mathfrak{P}}$, where $L_{\mathfrak{P}}$ is the completion of $L$ at $\mathfrak{P}$ (see~\cite{Dembele2}). By construction, $L_{\mathfrak{P}}$ is the compositum of the completions of $L_i$ at $\fp_i=\Z_{L_i}\cap\mathfrak{P}$, $i=1,\ldots, 5$. The form $g$ is ordinary at $\fp_5$ since $a_{\fp_5}(g)\equiv 1\pmod{5}$.  Therefore, by Wiles~\cite[Theorem 2]{Wiles}, we have
\[ \overline{\rho}_g^{(i)}|_{D_{\fp_5}}\sim\begin{pmatrix}\chi&*\\ 0& 1\end{pmatrix} \] 
for $i=1,\ldots,5$, where $D_{\fp_5}$ is the decomposition group at $\fp_5$ and $\chi$ is the mod $5$ cyclotomic character.  From this and the fact that $\Gal(L/F)=2\cdot\PSL_2(\F_5)^5$, it follows that $L_{\mathfrak{P}}=L_0(x_1^{1/5},\ldots, x_{5}^{1/5})$, where $L_0=F_{\fp_5}(\zeta_5)=\Q_5(\zeta_{25})$ and $x_i\in L_0^\times/\left(L_0^\times\right)^5$.  Note that $L_0$ is a tamely ramified extension of $F_{\fp_5}$ of degree $4$ and $d_{L_0}=5^{35}$.  Since $L_{\mathfrak{P}}/L_0$ is a Galois $5$-elementary abelian extension, the result of Moon~\cite[Lemma 2.1]{Moon1} gives that the different $\mathcal{D}_{L_\mathfrak{P}/L_0}$ divides $5^c$, where 
\[ c=(1+6/20)(1-1/5^5)=20306/15625. \] 
Therefore, we obtain that
\[ \delta_L\le 5^{35/20}\cdot 5^c \approx 135.384. \]
In fact, the exact root discriminant of $L$ is computed in~\cite{Roberts5} to be $\delta_L=125\cdot 5^{-1/12500}\approx 124.984$.  By a similar computation, we find that $\delta_K, \delta_{K'} \leq 135.39$ and $\delta_M \leq 5^{35/20+(1+6/20)(1-1/5^{10})} \leq 135.48$.
\end{rmk}

% By Fontaine~\cite[Remark 3.3.3]{Fontaine1}, this also provides an upper bound for the root discriminants $\delta_K,\delta_{K'}$ and $\delta_M$ which are independent of the dimensions of the underlying abelian varieties we used to construction $K$, $K'$ and $M$.  

\subsection*{Degree 10}

We conclude this section with the results of computations for the field $F=\Q(\zeta_{25})^+$, perfomed by Steve Donnelly.  The space of Hilbert modular cusp forms $S_2(1)$ of level $1$ for $F=\Q(\zeta_{25})^+$ has dimension $171$; we compute with $S_2(1)$ using the totally definite quaternion algebra $B$ over $F$ ramified at all $10$ real places of $F$ and no finite place.  The space $S_2(1)$ decomposes into irreducible Hecke constituents of dimensions 2, 4, 15 and 150 respectively.  We consider the Galois representation (constructed by Taylor \cite{taylor1}) 
$$\overline{\rho}_{\T\otimes\F_5}:\,\mathrm{Gal}(\overline{F}/F)\to\mathrm{GL}_2(\T\otimes\F_5).$$ 
The $\F_5$-algebra $\T\otimes\F_5$ has 17 non-Eisenstein maximal ideals. In Table \ref{tabled10}, we group these maximal ideals according to the degree of their residue fields and note the action of $\Gal(F/\Q)$, generated by $\sigma$.

\begin{table}[h]
\begin{equation} \label{tabled10} \notag 
\begin{array}{c||c|c}
\text{Inertial degree}&\text{Number of ideals}&\text{Galois action}\\ \hline
2\rule{0pt}{2.5ex} &10&\text{$\sigma$ permutes} \\
5&2&\text{$\sigma$ permutes, $\sigma^2$ fixes} \\
10&1& \text{$\sigma$ fixes} \\
15&1& \text{$\sigma^2$ fixes, but not $\sigma$} \\
25&2& \text{$\sigma$ permutes, $\sigma^2$ fixes} \\
40&1& \text{$\sigma$ fixes}
\end{array}
\end{equation}
\begin{center}
\textbf{Table \ref{tabled10}}: Hecke data for the representation $\overline{\rho}$ of level $1$ for $F=\Q(\zeta_{25})^+$
\end{center}
\end{table}
\addtocounter{equation}{1}

Arguing as above, we find Galois extensions ramified only at $5$ with Galois groups
\begin{center} 
$10 \cdot 2 \cdot \PSL_2(\F_{5^2})^{10},\ \ 10\cdot 2\cdot \PSL_2(\F_{5^5})^2,\ \  10\cdot \PGL_2(\F_{5^{10}})$, \\
$10\cdot \PGL_2(\F_{5^{15}}),\ \  10\cdot 2 \cdot \PSL_2(\F_{5^{25}})^2,\ \ 10\cdot \PGL_2(\F_{5^{40}})$,
\end{center}
respectively. By the result of Fontaine~\cite[Corollaire 3.3.2]{Fontaine1}, we find that the root discriminants of these fields are bounded by $5^{17/10+(1+1/4)} \leq 115.34$. 

\section{A nonsolvable number field ramified at $7$?}

To conclude our paper, we discuss the computational obstacles that we confront in applying the above techniques when $p=7$.

In this case, the possible base fields are the subfields of $\Q(\zeta_{49})$ with degrees $3$ and $7$.  First, we consider the field $F=\Q(\zeta_7)^+$.  At level $1$, the group $\Gamma(1)$ we encounter is the celebrated $(2,3,7)$-triangle group \cite{ElkiesShim}; the curve $X(1)$ is the Fuchsian group with the smallest possible area, and torsion-free (normal) subgroups of $\Gamma(1)$ are associated to Hurwitz curves, i.e., curves of genus $g$ with the largest possible order $84(g-1)$ of their automorphism group.  In particular, $X(1)$ has genus $0$, as does the curve $X_0(\fp_7)$ for $\fp_7 \mid 7$.  At level $\fp_7^2$, we find a genus $1$ curve $X_0(\fp_7^2)$ (in fact, an elliptic curve since $\Q(\zeta_7)$ has class number $1$ so there is an $F$-rational CM point on $X$) which corresponds to the base change to $F$ of the classical modular curve $X_0(49)$, or equivalently (up to isogeny) the elliptic curve over $\Q$ with complex multiplication by $\Z[(1+\sqrt{-7})/2]$ and $j$-invariant $-3375$.  As for $p=3$, it then follows from Buzzard, Diamond and Jarvis~\cite[Proposition 4.13(a)]{BDJ} that there are only reducible forms in levels which are higher powers of $\fp_7$.

Next, let $F$ be the subfield of degree $7$ in $\Q(\zeta_{49})$ with discriminant $d_F=13841287201=7^{12}$.  The corresponding Shimura curve has area $275381/6$ and signature $(22864;2^{71}, 3^{203})$, and this places it well outside the realm of practical computations using the above techniques.  We must therefore leave the problem of finding a nonsolvable Galois number field ramified only at $p=7$ for future investigation---it is conceivable that a good choice of group and base field, as explained in the introduction, will exhibit such a field.

One possible alternative track, mimicking the fortunate accident for $p=5$ exploited by Roberts, would be to search for an elliptic curve with good reduction away from $7$ defined over a number field (of any signature) which is ramified only at $7$.  We searched for curves over $\Q(\zeta_7)$ using the method of Cremona-Lingham \cite{CremonaLingham}, but our search was not exhaustive (due to the complexity of finding $S$-integral points on the relevant discriminant elliptic curves over number fields), and consequently we only managed to recover the base changes of curves defined over $\Q$ (yielding solvable extensions).  

We note finally that if one considers the weaker question of exhibiting a nonsolvable Galois number field which is unramified outside a finite set $S$ of primes with $7 \in S$, such as $S=\{2,7\}$ or $S=\{3,7\}$, then already one can find such extensions by classical forms of weight $2$ over $\Q$: see work of Wiese~\cite[Theorem 1.1]{Wiese} for a much more general result.

\end{document}